\documentclass[12 pt, reqno]{amsart}
\usepackage{amsmath,times,epsfig,amssymb,amsbsy,amscd,amsfonts,amstext,graphicx,color,bm,amsthm}
\usepackage[all]{xy}
\usepackage{graphicx}

\usepackage{wrapfig}   %-|    for option landscape
\usepackage{lscape}      %|
\usepackage{rotating}    %|
\usepackage{epstopdf} %-|

\usepackage{tikz,subfigure}
\usepackage{tikz-cd}

\newcommand{\KK}{\mathbb{K}}

\newcommand{\CC}{\mathbb{C}}
\newcommand{\A}{\mathcal{A}}

\newcommand \ideal[1] {\langle #1 \rangle}

\newtheorem{Theorem}{Theorem}[section]
\newtheorem{Definition}[Theorem]{Definition}

\newtheorem{Lemma}[Theorem]{Lemma}
\newtheorem{Proposition}[Theorem]{Proposition}
\newtheorem{Corollary}[Theorem]{Corollary}
\newtheorem{Remark}[Theorem]{Remark}
\newtheorem{Example}[Theorem]{Example}

\DeclareMathOperator{\POexp}{POexp}
\DeclareMathOperator{\depth}{depth}
\DeclareMathOperator{\Ass}{Ass}
\DeclareMathOperator{\rk}{rk}
\DeclareMathOperator{\Der}{Der}
\DeclareMathOperator{\projdim}{projdim}
\DeclareMathOperator{\pdeg}{pdeg}
\DeclareMathOperator{\codim}{codim}

\begin{document}

\title{Localization of plus-one generated arrangements}

\begin{abstract}
We study the classes of free and plus-one generated hyperplane arrangements. Specifically, we describe how to compute the associated prime ideals of the Jacobian ideal of such an arrangement from its lattice of intersection. Moreover, we prove that the localization of a plus-one generated arrangement is free or plus-one generated. 
\\[0.2em]


\noindent
MSC10: 52C35, 32S22.
\end{abstract}

\author{Elisa Palezzato}
\address{Elisa Palezzato, Department of Mathematics, Hokkaido University, Kita 10, Nishi 8, Kita-Ku, Sapporo 060-0810, Japan.}
\email{palezzato@math.sci.hokudai.ac.jp}
\author{Michele Torielli}
\address{Michele Torielli, Department of Mathematics, GI-CoRE GSB, Hokkaido University, Kita 10, Nishi 8, Kita-Ku, Sapporo 060-0810, Japan.}
\email{torielli@math.sci.hokudai.ac.jp}

%\dedicatory{}

\date{\today}
\maketitle

%\tableofcontents

%---------------------------------------------%

\section{Introduction}

Let $V$ be a vector space of dimension $l$ over a field $\KK$. Fix a system of coordinates $(x_1,\dots, x_l)$ of $V^\ast$. 
We denote by $S = S(V^\ast) = \KK[x_1,\dots, x_l]$ the symmetric algebra of $V^\ast$. 
A hyperplane arrangement $\A = \{H_1, \dots, H_n\}$ is a finite collection of hyperplanes in $V$. We refer to \cite{orlterao}
as main reference on the theory of arrangements.

In the theory of hyperplane arrangements, the freeness is a very important algebraic property. In fact, freeness implies several interesting geometric and combinatorial properties of the arrangement itself, see \cite{orlterao}. By definition, an arrangement is free if and only if its module of logarithmic derivations is a free module. A lot is known about free arrangements, however there is still some mystery around the notion of freeness. For example, Terao's conjecture asserting the dependence of freeness only on the combinatorics is the longstanding open problem in this area. 

In order to study this conjecture, in \cite{abe2018plus} Abe introduced the notion of plus-one generated arrangement, where an arrangement is plus-one generated if and only if its module of logarithmic derivations is generated by $l+1$ elements and we can ``control'' their first syzygy. Moreover, Abe described how free and plus-one generated arrangements are connected. In particular, he proved that the deletion of a free arrangement is free or plus-one generated and vice versa, under certain additional hypothesis, if the deletion is plus-one generated then the original arrangement is free or plus-one generated.

The goal of this article is to study more in depth these two classes of arrangements. Specifically, given an arrangement $\A$, we describe how to compute the associated prime ideals of the module $S/J(\A)$ from its lattice of intersection $L(\A)$, where $J(\A)$ is the Jacobian ideal of $\A$. Moreover, we prove that the localization of a plus-one generated arrangement is free or plus-one generated. Finally, we describe an example which demonstrates that the deletion of a plus-one generated arrangement is not necessarily free or plus-one generated.

All the computations in this article have been performed using the software CoCoA, see \cite{palezzato2018hyperplane}.

\section{Preliminares on hyperplane arrangements}\label{sec:arr}

In this section, we recall the terminology, the basic notations and some fundamental results related to hyperplane arrangements.

%Let $V$ be an $l$-dimensional vector space over a field $K$. 
Let $\KK$ be a field of characteristic zero. A finite set of affine hyperplanes $\A =\{H_1, \dots, H_n\}$ in $\KK^l$ 
is called a \textbf{hyperplane arrangement}. For each hyperplane $H_i$ we fix a defining polynomial $\alpha_i\in S= \KK[x_1,\dots, x_l]$ such that $H_i = \alpha_i^{-1}(0)$, 
and let $Q(\A)=\prod_{i=1}^n\alpha_i$. An arrangement $\A$ is called \textbf{central} if each $H_i$ contains the origin of $\KK^l$. 
In this case, each $\alpha_i\in S$ is a linear homogeneous polynomial, and hence $Q(\A)$ is homogeneous of degree $n$. 
The operation of \textbf{coning} allows one to transform any arrangement $\A$ in $\KK^l$ with $n$ hyperplanes into a central arrangement $c\A$ with $n + 1$ hyperplanes in $\KK^{l+1}$, see \cite{orlterao}.
Unless otherwise specified, we will only consider central hyperplane arrangements. For this reason every time we will study $\A$ a central hyperplane arrangement, we will omit the word central.

Let $L(\A)=\{\bigcap_{H\in\mathcal{B}}H \mid \mathcal{B}\subseteq\A\}$ be the \textbf{lattice of intersection} of $\A$, %, i.e. the set of non-empty intersections of elements of all subsets of $\A$. 
ordered by reverse inclusion, i.e. $X\le Y$ if and only if $Y\subseteq X$, for $X,Y\in L(\A)$.
%Define a partial order on $L(\A)$ by $X\le Y$ if and only if $Y\subseteq X$, for all $X,Y\in L(\A)$. 
%Note that this is the reverse inclusion. 
Define a rank function on $L(\A)$ by $\rk(X)=\codim(X)$. 
$L(\A)$ plays a fundamental role in the study of hyperplane arrangements, in fact it determines the combinatorics of the arrangement.
Let $L(\A)_p=\{X\in L(\A)~|~\rk(X)=p\}$. We call $\A$ \textbf{essential} if $L(\A)_l\ne\emptyset$.

For any flat $X\in L(\A)$ define the \textbf{localization} of $\A$ to $X$ as the subarrangement $\A_X$ of $\A$ by 
$$\A_X=\{H\in\A~|~X\subseteq H\}.$$

The \textbf{restriction} of $\A$ to $H\in\A$ is the arrangement $\A^H$ in $H\cong \KK^{l-1}$ defined by
$$\A^H =\{H\cap H'~|~H'\in\A\setminus\{H\} \text{ and } H\cap H'\ne\emptyset\}.$$

%\section{Free hyperplane arrangements}
%We recall the basic notions and properties of free hyperplane arrangements.

%Let $V = \C^l$ be a complex vector space with coordinate $(x_1,\dots ,x_l)$, $S = S(V^\ast) = \C[x_1,\dots,x_l]$ 
%the polynomial ring and $\A = \{H_1,\dots,H_n\}$ be a essential hyperplane arrangement. 
We denote by $\Der_{\KK^l} =\{\sum_{i=1}^l f_i\partial_{x_i}~|~f_i\in S\}$ the $S$-module of \textbf{polynomial vector fields} on $\KK^l$ (or $S$-derivations). 
Let $\delta =  \sum_{i=1}^l f_i\partial_{x_i}\in \Der_{\KK^l}$. Then $\delta$ is  said to be \textbf{homogeneous of polynomial degree} $d$ if $f_1, \dots, f_l$ are homogeneous polynomials of degree~$d$ in $S$. 
In this case, we write $\pdeg(\delta) = d$.

\begin{Definition} 
Let $\A$ be an arrangement in $\KK^l$. Define the \textbf{module of vector fields logarithmic tangent} to $\A$ (or logarithmic vector fields) by
$$D(\A) = \{\delta\in \Der_{\KK^l}~|~ \delta(\alpha_i) \in \ideal{\alpha_i} S, \forall i\}.$$
\end{Definition}

The module $D(\A)$ is obviously a graded $S$-module and we have that $$D(\A)= \{\delta\in \Der_{\KK^l}~|~ \delta(Q(\A)) \in \ideal{Q(\A)} S\}.$$ 

In particular, since the arrangement $\A$ is central, then the Euler vector field $\delta_E=\sum_{i=1}^lx_i\partial_{x_i}$ belongs to $D(\A)$, in fact
$\delta_E(Q(\A))=nQ(\A)$. 
In this case, we can write $D(\A)\cong S{\cdot}\delta_E\oplus D_0(\A)$, where $$D_0(\A)=\{\delta\in \Der_{\KK^l}~|~ \delta(Q(\A))=0\}.$$

\begin{Definition} 
An arrangement $\A$ in $\KK^l$ is said to be \textbf{free with exponents $(e_1,\dots,e_l)$} 
if and only if $D(\A)$ is a free $S$-module and there exists a basis $\delta_1,\dots,\delta_l \in D(\A)$ 
such that $\pdeg(\delta_i) = e_i$, or equivalently $D(\A)\cong\bigoplus_{i=1}^lS(-e_i)$.
\end{Definition}
In the rest of the paper, given a tuple of integers $(e_1,\dots,e_l)$, we will write $(e_1,\dots,e_l)_{\le}$, if we assume that $e_1\le e_2\le\cdots\le e_l$.

%Let $\delta_1,\dots,\delta_l \in D(\A)$.  Then $\det(\delta_i(x_j))_{i,j}$ is divisible by $Q(\A)$.
One of the most famous characterizations of freeness is due to Saito \cite{saito} and it uses  the determinant of the coefficient 
matrix of $\delta_1,\dots,\delta_l$ to check if the arrangement $\A$ is free or not. %Notice that the original statement is for
%characteristic $0$, but in \cite{terao1983free} Terao showed that this statement holds true for any characteristic.

\begin{Theorem}[Saito's criterion]\label{theo:saitocrit}
Let $\A$ be an arrangement in $\KK^l$ and $\delta_1, \dots, \delta_l \in D(\A)$. Then the following facts are equivalent
\begin{enumerate}
\item $D(\A)$ is free with basis $\delta_1, \dots, \delta_l$, i. e. $D(\A) = S\cdot\delta_1\oplus \cdots \oplus S \cdot\delta_l$.
\item $\det(\delta_i(x_j))_{i,j}=c Q(\A)$, where $c\in \KK\setminus\{0\}$.
\item $\delta_1, \dots, \delta_l$ are linearly independent over $S$ and $\sum_{i=1}^l\pdeg(\delta_i)=n$.
\end{enumerate}
\end{Theorem}

Using Saito's criterion one can prove the following result that will play an important role in Sections 4 and 5. Moreover, we will generalize it in Theorem~\ref{theo:localofplusoneisfreeplusone} to the case of plus-one generated arrangements. 
\begin{Theorem}[{\cite[Theorem 4.37]{orlterao}}]\label{theo:localisfree} If $\A$ is free, then $\A_X$ is free for any $X\in L(\A)$.
\end{Theorem}

Given an arrangement $\A$ in $\KK^l$, the \textbf{Jacobian ideal} of $\A$
is the ideal of $S$ generated by $Q(\A)$ and all its partial derivatives, and it is denoted by $J(\A)$.
$D_0(\A)$ can be identified with the first syzygies of $J(\A)$. In particular we have the following exact sequence
$$0\to D_0(\A)\to S(-n+1)^l\to J(\A)\to0.$$

%Given an arrangement $\A$ in $\KK^l$, the \textbf{Jacobian ideal} of $\A$
%is the ideal of $S$ generated by $Q(\A)$ and all its partial derivatives, and it is denoted by $J(\A)$.
The Jacobian ideal has a central role in the study of free arrangements, see \cite{orlterao}, \cite{Gin-freearr} and \cite{LefschetzPalez} for more details.
In fact, we can also characterize freeness by looking at $J(\A)$ via the Terao's criterion.
Notice that Terao described this result for characteristic $0$, but the statement holds true for any characteristic as shown in \cite{palezzato2018free}.

\begin{Theorem}[\cite{terao1980arrangementsI}]\label{theo:freCMcod2} 
An arrangement $\A$ in $\KK^l$ is free if and only if $S/J(\A)$ is $0$ or $(l-2)$-dimensional Cohen--Macaulay.
\end{Theorem}

\section{Plus-one generated arrangements}

In \cite{abe2018plus}, the author generalized the notions of free and nearly free (see \cite{dimca2015nearly}) to central and essential arrangements in any dimension. However, the definition can be given also for non-essential arrangements.

%\begin{Definition}\label{def:plusonegen}
%Let $\A=\{H_1,\dots,H_n\}$ be a central and essential arrangement in $\KK^l$. We say that $\A$ is \textbf{plus-one generated} with \textbf{exponents} $\POexp(\A)=(d_1,\dots,d_l)$ and \textbf{level} $d$ if $D(\A)$ has a minimal free resolution of the following form
%\begin{equation*}\label{eq:defPOGarr}
%\xymatrix{0 \ar[r] & S(-d-1) \ar[rr]^-{(\alpha,f_1,\dots,f_l)} && S(-d) {\oplus}(\bigoplus_{i=1}^l S(-d_i)) \ar[r] & D(\A) \ar[r] & 0 }
%\end{equation*}
%We say that $\A$ is \textbf{strictly plus-one generated} with \textbf{exponents} $\POexp(\A)=(d_1,\dots,d_l)$ and \textbf{level} $d$ if $\A$ is plus-one generated with the same exponents and level, and $\alpha\ne0$ in the previous resolution of $D(\A)$. 
%\end{Definition}

\begin{Definition}\label{def:plusonegen}
Let $\A=\{H_1,\dots,H_n\}$ be an arrangement in $\KK^l$. We say that $\A$ is \textbf{plus-one generated} with \textbf{exponents} $\POexp(\A)=(d_1,\dots,d_l)$ and \textbf{level} $d$ if $D(\A)$ has a minimal free resolution of the following form
\begin{equation}\label{eq:respogDA}
\xymatrix{0 \ar[r] & S(-d-1) \ar[rr]^-{(\alpha,f_1,\dots,f_l)} && S(-d) {\oplus}(\bigoplus_{i=1}^l S(-d_i)) \ar[r] & D(\A) \ar[r] & 0 }
\end{equation}
\end{Definition}

\begin{Remark} Let $\A$ be a plus-one generated arrangement in $\KK^l$ with exponents $\POexp(\A)=(d_1,\dots,d_l)_{\le}$ and level $d$. Since $\A$ is central, then there exists $k\ge2$ such that $(d_1,\dots,d_l)_{\le}=(0,\dots,0,1,d_k,\dots,d_l)_{\le}$. If $\A$ is essential, then $k=2$. If $\A$ is non-essential, then $k\ge3$.
\end{Remark}

Directly from the definition, we can show the following

\begin{Lemma}\label{lemma:fromresDAtoJac}
Let $\A=\{H_1,\dots,H_n\}$ be an arrangement in $\KK^l$. $\A$ is plus-one generated with exponents $\POexp(\A)=(d_1,\dots,d_l)_{\le}$ and level $d$
if and only if $S/J(\A)$ has a minimal free resolution of the form
%\begin{equation*}
%0{\to}S(-n-d){\to} S(-n-d+1){\oplus}(\bigoplus_{i=2}^{l} S(-n-d_i+1)) {\to} S(-n+1)^l{\to}S
%\end{equation*}
%Moreover, the map 
%$$\partial_3\colon S(-n-d)\to S(-n-d+1)\oplus(\bigoplus_{i=2}^{l} S(-n-d_i+1))$$
%is defined by a matrix of the form $(\alpha, f_{i_1},\dots,f_{i_{l-1}})$, where $1\le i_1<\cdots<i_{l-1}\le l$.
\begin{equation}\label{eq:respogjac}
0{\to}S(-n-d){\to} S(-n-d+1){\oplus}(\bigoplus_{i=k}^{l} S(-n-d_i+1)) {\to} S(-n+1)^{l-k+2}{\to}S.
\end{equation}
Moreover, the map 
$$\partial_3\colon S(-n-d)\to S(-n-d+1)\oplus(\bigoplus_{i=k}^{l} S(-n-d_i+1))$$
is defined by a matrix of the form $(\alpha, f_{i_1},\dots,f_{i_{l-k+1}})$, where $1\le i_1<\cdots<i_{l-k+1}\le l$.
Notice that $l-k+2$ coincides with the codimension of the center of $\A$ or, equivalently, the rank of $\A$.
\end{Lemma}
\begin{proof}
This equivalence follows from the fact that $D(\A)\cong S\cdot \delta_E \oplus D_0(\A)$, and the fact that if $\delta_1,\dots,\delta_l\in D_0(\A)$ and there is a relation of the form
$g_0\delta_E+\sum_{i=1}^lg_i\delta_i=0$, with $g_i\in  S$, then $g_0=0$.
\end{proof}

\begin{Example}\label{EX:Arr in the picture}
Consider the arrangement $\A$ in $\CC^4$ with defining polynomial $x(x-y)(x-t)(y-z)(z-t)$. % whose lattice is described in Figure \ref{FIG:lat6}. 
It is plus-one generated with $\POexp(\A)=(1,1,2,2)$ and level $2$. In fact $D(\A)$ has a minimal resolution of the form
\begin{equation*}
0\to S(-3)\to S(-1)^2\oplus S(-2)^3 \to D(\A).
\end{equation*}
On the other hand, the minimal resolution of $S/J(\A)$ is
\begin{equation*}
0\to S(-7)\to S(-5)\oplus S(-6)^3 \to S(-4)^4\to S.
\end{equation*}
\end{Example}

\begin{Example}\label{EX:ShiD}
Consider the arrangement $\A$ in $\CC^4$ with defining polynomial $(x+y)(x-z)(x+z)(y-z)(y+z)(x-y-t)(x+y-t)(x-z-t)(x+z-t)(y-z-t)(y+z-t)t$.
This arrangement is plus-one generated with $\POexp(\A)=(1,4,4,4)$ and level $5$ since the minimal resolution of $S/J(\A)$ is
\begin{equation*}
0\to S(-17)\to S(-15)^3\oplus S(-16)\to S(-11)^4\to S.
\end{equation*}
\end{Example}

Notice that for an arrangement $\A$ to be plus-one generated it is not enough that $D(\A)$ has projective dimension $1$ or equivalently that $S/J(\A)$ has projective dimension $3$.

\begin{Example}\label{EX:noPO}
Consider the arrangement $\A$ in $\CC^4$ with defining polynomial $xyzt(x+y-2z)(x-3y+z)(-5x+y+z)(x+y+z)$. It is not a plus-one generated  arrangement since the minimal resolution of $S/J(\A)$ is
\begin{equation*}
0 \to S(-13)\to S(-8)\oplus S(-11)^3\to S(-7)^4\to S.
\end{equation*}
\end{Example}

In \cite{abe2018plus}, Abe also described how free and plus-one generated arrangements are connected.

\begin{Theorem}[{\cite[Theorem 1.4]{abe2018plus}}]\label{theo:delisfreepog} Let $\A$ be an essential free hyperplane arrangement with exponents $(e_1, \ldots, e_l)$ and $H\in\A$. Then $\A\setminus\{H\}$ is free, or plus-one generated with exponents $(e_1, \ldots, e_l)$ and level $|\A\setminus\{H\}|-|\A^H|$.
\end{Theorem}

\begin{Theorem}[{\cite[Theorem 1.9]{abe2018plus}}] Let $\A$ be an essential arrangement and $H\in\A$. Assume that $\A \setminus \{H\}$ is free with exponents $(e_1,\dots, e_l)_{\le}$. If $|\A \setminus \{H\}|-|\A^H|\ge e_{l-2}$, then $\A$ is free, or plus-one generated with $\POexp(\A) =$ $ (e_1, \dots, e_{l-2}, e_{l-1} + 1, e_l + 1)$ and level $e_{l-1} +e_l - |\A|+|\A^H|+1$.
\end{Theorem}

%---------------------------------------------%

\section{Associated prime ideals of $S/J(\A)$}
Let $\A$ be an arrangement in $\KK^l$ and $k\ge 2$. To each $X\in L(\A)_k$ corresponds a prime ideal $I(X)=\ideal{\alpha_{i_1},\dots,\alpha_{i_k}}$ of codimension $k$ of $S=\KK[x_1,\dots,x_l]$ that contains $J(\A)$, where $X=H_{i_1}\cap\cdots\cap H_{i_k}$. %We denote such ideal by $I(X)$.

%Notice that if $X$ is in $L(\A)_k$, then $I(X)=\ideal{\alpha_{i_1},\dots,\alpha_{i_k}}$, where $\alpha_{1},\dots,\alpha_{n}$ are the defining polynomials of the hyperplanes in $\A$.

\begin{Remark}\label{REM:EmbeddedPrime}
Let $\A$ be an arrangement in $\KK^l$. There is a bijection between $L(\A)_2$ and the set of associated prime ideals of $S/J(\A)$ of codimension $2$. This is because each associated prime ideal of $S/J(\A)$ of codimension $2$ corresponds to an irreducible component of the singular locus of $\A$. Moreover, any associated prime ideal of $S/J(\A)$ of codimention $k$ corresponds to an element of $L(\A)_k$. (For details see Section~3.8 of \cite{eisenbud}). This gives us the following inclusions
$$\{I(X) ~|~ X\in L(\A)_2\}\subseteq\Ass_S(S/J(\A))\subseteq \{I(X) ~|~ X\in L(\A)_{\ge2}\}.$$
\end{Remark}

%As noted in Remark~\ref{REM:EmbeddedPrime}, $\{I(X) ~|~ X\in L(\A)_2\}\subseteq\Ass_S(S/J(\A))\subseteq \{I(X) ~|~ X\in L(\A)_k \textit{ with } k\ge2\}$.

\begin{Theorem}\label{THM:AssPrimeAndLocal}
Let $\A$ be an arrangement in $\KK^l$, and $X,Y\in L(\A)$ such that $X\subseteq Y$.
$$I(Y)\in\Ass_S(S/J(\A)) \Longleftrightarrow I(Y)\in\Ass_S(S/J(\A_X)).$$
\end{Theorem}
\begin{proof}
%Since $\A_X$ is a subarrangement of $\A$, then $J(\A_X)\supseteq J(\A)$. This implies that we have an injective map of $S$-modules $S/J(\A_X) \hookrightarrow S/J(\A)$. By definition of associated prime ideals we have that $\Ass_S(S/J(\A_X)) \subseteq \Ass_S(S/J(\A))$. (For more details see Lemma~3.6~\cite{eisenbud2013commutative}).
%Conversely, if we consider $S_{I(X)}$ the localization of the ring $S$ by the ideal $I(X)$ we obtain that $J(\A)S_{I(X)}=J(\A_X)S_{I(X)}$. Moreover, by Theorem~3.1~\cite{eisenbud2013commutative}, if $I(Y)\in \Ass_S(S/J(\A))$, then $I(Y)S_{I(X)} \in \Ass_{S_{I(X)}}(S_{I(X)}/J(\A)S_{I(X)})=\Ass_{S_{I(X)}}(S_{I(X)}/J(\A_X)S_{I(X)})$. Since the ideals of $S_{I(X)}$ are in bijection with the ideal of $S$ contained in $I(X)$, we have that $I(Y)S_{I(X)} \in \Ass_{S_{I(X)}}(S_{I(X)}/J(\A_X)S_{I(X)})$ if and only if $I(Y)\in\Ass_S(S/J(\A_X)).$
%Since $\A_X$ is a subarrangement of $\A$, then $J(\A_X)\supseteq J(\A)$. This implies that we have an injective map of $S$-modules $S/J(\A_X) \hookrightarrow S/J(\A)$. By definition of associated prime ideals we have that $\Ass_S(S/J(\A_X)) \subseteq \Ass_S(S/J(\A))$. (For more details see Lemma~3.6~\cite{eisenbud2013commutative}).
Consider $S_{I(X)}$ the localization of the ring $S$ by the ideal $I(X)$. We have that $J(\A)S_{I(X)}=J(\A_X)S_{I(X)}$. Since the ideals of $S_{I(X)}$ are in bijection with the ideal of $S$ contained in $I(X)$, by \cite[Theorem~3.1]{eisenbud}, $I(Y)\in \Ass_S(S/J(\A))$ if and only if $I(Y)S_{I(X)} \in \Ass_{S_{I(X)}}(S_{I(X)}/J(\A)S_{I(X)})=\Ass_{S_{I(X)}}(S_{I(X)}/J(\A_X)S_{I(X)})$ if and only if $I(Y)\in\Ass_S(S/J(\A_X)).$
\end{proof}

\begin{Corollary}\label{COR:AssPrimeInclusion}
Let $\A$ be an arrangement in $\KK^l$. Then 
$$\Ass_S(S/J(\A)){\subseteq} \{I(X) ~|~ X{\in} L(\A)_2\}{\cup}\{I(X) ~|~ X{\in} L(\A)_{\ge3}, \A_X \textit{ non-free} \}.$$
\end{Corollary}
\begin{proof}
From Remark~\ref{REM:EmbeddedPrime}, every ideal in $\Ass_S(S/J(\A))$ is of the form $I(X)$ for some $X\in L(\A)_{\ge2}$. If $X\in L(\A)_{\ge3}$, then by Theorem~\ref{THM:AssPrimeAndLocal} $I(X)\in \Ass_S(S/J(\A_X))$. This implies that $S/J(\A_X)$ is not Cohen--Macaulay and hence $\A_X$ is non-free.
\end{proof}

The previous inclusion might not be an equality in general.
\begin{Example}[{c.f. \cite[Example 8.5]{dimca2017hyperplane}}] \label{EX:DimcaNONonly2flats}
Consider the arrangement $\A$ in $\CC^4$ with defining polynomial $Q(\A)=\prod_{a=(a_0,a_1,a_2,a_3)} (a_0x_0+a_1x_1+a_2x_2+a_3x_3)$, where $a\in\{0,1\}^4$ and $a\ne(0,0,0,0)$. This is plus-one generated with $\POexp(\A)=(1,5,5,5)$ and level $5$. In fact, $S/J(\A)$ has minimal resolution
\begin{equation*}
0\to S(-20)\to S(-19)^4\to S(-14)^4\to S.
\end{equation*}
A direct computation shows that $\Ass_S(S/J(\A))= \{I(X) ~|~ X\in L(\A)_2\}$. However, if we consider $X\in L(\A)_4$ the intersection of the hyperplanes $x_1=0,x_2=0,x_3=0,x_4=0$, then $\A_X=\A$, and hence $\A_X$ is not free.
\end{Example}

%\begin{Example}
%Consider $\A$ the cone of the arrangement in Example~\ref{EX:DimcaNONonly2flats}. Then $\A$ is a non-free arrangement such that $\Ass_S(S/J(\A))= \{I(X) ~|~ X\in L(\A)_2\}$. However, if we consider $X\in L(\A)_4$ the intersection of the hyperplanes $x_1=0,x_2=0,x_3=0,x_4=0$, then $\A_X$ coincides with the arrangement in Example~\ref{EX:DimcaNONonly2flats}, and hence $\A_X$ is not free.
%\end{Example}

Since it is known by Theorem~\ref{theo:localisfree} that if $\A$ is a free arrangement then the localization $\A_X$ is free for any $X\in L(\A)$, by Theorem~\ref{THM:AssPrimeAndLocal} and Remark~\ref{REM:EmbeddedPrime}, we have the following.

\begin{Corollary}\label{Corol:assprimefreearr}
Let $\A$ be a free arrangement. Then 
$$\Ass_S(S/J(\A)) = \{I(X) ~|~ X\in L(\A)_2\}.$$
\end{Corollary}
\begin{Remark} By Example~\ref{EX:DimcaNONonly2flats}, the statement of Corollary~\ref{Corol:assprimefreearr} is not an equivalence.
\end{Remark}

We are now ready to state the main result of this section.
\begin{Theorem}\label{Theo:AssPrimeIdealNonFreeArr}
Let $\A$ be an arrangement in $\KK^l$ such that $S/J(\A)$ has projective dimension $3$. Then 
$$\Ass_S(S/J(\A)) = \{I(X) ~|~ X{\in} L(\A)_2\}{\cup}\{I(X) ~|~ X{\in} L(\A)_3, \A_X \textit{ non-free} \}.$$
In particular, this holds if $\A$ is a plus-one generated arrangement.
\end{Theorem}
\begin{proof}
By Auslander--Buchsbaum formula (\cite[Theorem~19.9]{eisenbud}), we have that $\depth(S/J(\A))=l-3$. Since the depth of a module is bounded above by the dimension of its associated prime ideals, we have that $S/J(\A)$ cannot have associated prime ideals of codimension $k$ with $k\ge4$. %This implies that each ideal in $\Ass_S(S/J(\A))$ has codimension at most $3$. 
From Corollary~\ref{COR:AssPrimeInclusion}, we have the inclusion ``$\subseteq$".
Let $X\in L(\A)_3$ be such that $\A_X$ is non-free. Since $\A_X$ is an arrangement of rank $3$, $J(\A_X)$ coincides with its saturation with respect to the ideal $I(X)$ if and only if $\A_X$ is free, as described in the introduction of \cite{dimca2015nearly}. Since we assume that $\A_X$ is non-free, this implies that $I(X)$ is an associated prime ideal of $S/J(\A_X)$. By Theorem~\ref{THM:AssPrimeAndLocal}, $I(X)\in \Ass_S(S/J(\A))$.
\end{proof}

\begin{Example}
Consider the hyperplane arrangement in Example~\ref{EX:Arr in the picture}. Then the associated prime ideals of $S/J(\A)$ correspond to all rank $2$ flats in $L(\A)$ and the rank $3$ flat corresponding to the intersection of the hyperplanes with equation $y=z$, $x=t$ and $z=t$.
%Notice that this is the only rank $3$ flat in $L(\A)$ such that the number of rank $2$ flats containing it is maximal. 
\end{Example}

\begin{Example}\label{EX:ShiDbis}
Consider the hyperplane arrangement in Example~\ref{EX:ShiD}. Then the associated prime ideals of $S/J(\A)$ correspond to all rank $2$ flats in $L(\A)$ and the ideals $\ideal{y+z, x+z,t}$, $\ideal{y-z, x-z,t}$,  and $\ideal{y-t,x-t,z}$.
%Notice that there are four rank $3$ flats in $L(\A)$ such that the number of rank $2$ flats containing them is maximal. However, not all of them correspond to an associated prime ideal of $M(Q)$. Between them, the ones that correspond to an associated prime ideal are the ones that are contained in the minimal number of hyperplanes.
\end{Example}

%The statement of Theorem~\ref{Theo:AssPrimeIdealNonFreeArr} is not an equivalence.
\begin{Example}
%The converse of Proposition~\ref{PROP:AssPrime} is not true in general. In fact, it is enough to 
Consider the arrangement $\A$ in $\CC^4$ with defining polynomial $t(x+y+z)(2x+4y+5z)(x+4y-5z)(-3x+5y+z)(2x+7y+2z)(3x-4y+9z)$. This is not plus-one generated because $S/J(\A)$ has a minimal free resolution
\begin{equation*}
0\to S(-11)^3\to S(-7) \oplus S(-10)^5\to S(-6)^4\to S.
\end{equation*}
However, $S/J(\A)$ has only $\ideal{x,y,z}$ as embedded prime ideal.
\end{Example}

%---------------------------------------------%

\section{Localization of plus-one generated arrangements}

The goal of this section is to describe the localization of plus-one generated arrangements.
The first step is to relate the algebraic total Betti numbers of an arrangement and the ones of its localization.

\begin{Proposition}\label{prop:localBettinumbers} Let $\A$ be an arrangement in $\KK^l$ and $X\in L(\A)$. Then for all $i\ge0$, we have $$b_i(S/J(\A_X))\le b_i(S/J(\A)) \text{ and }$$
$$b_i(D(\A_X))\le b_i(D(\A)),$$ where $b_i(-)$ are the algebraic total Betti numbers.
% the algebraic total Betti numbers of $S/J(\A_X)$ are smaller than  the algebraic total Betti numbers of $S/J(\A)$.
\end{Proposition}
\begin{proof} If $X\in L(\A)$ is such that $\A=\A_X$, the statement is obviously true. Assume $\A_X\subsetneq\A$. Without loss of generalities, we can make a change of coordinates and assume that $I(X)=(x_1, \dots, x_s)$ for some $1\le s\le l-1$. Consider $\mathcal{F}_{\A}$
%\begin{equation*}
%0\to S\to S^{l-k+1} \to S^{l-k+2}\to S\to S/J(\A)\to 0
%\end{equation*}
a minimal free resolution of $S/J(\A)$. As in the proof of Theorem~\ref{THM:AssPrimeAndLocal}, $J(\A)S_{I(X)}=J(\A_X)S_{I(X)}$, and hence $(S/J(\A))_{I(X)}\cong S_{I(X)}/J(\A)S_{I(X)}\cong S_{I(X)}/J(\A_X)S_{I(X)}.$ Since the localization preserves freeness, if we localize $\mathcal{F}_{\A}$ at the prime ideal $I(X)$ we obtain the exact sequence $(\mathcal{F}_{\A})_{I(X)}$.
%\begin{equation*}
%0\to S_{I(X)}\to S_{I(X)}^{l-k+1} \to S_{I(X)}^{l-k+2}\to S_{I(X)}\to S_{I(X)}/J(\A_X)S_{I(X)}\to 0.
%\end{equation*}
In general, $(\mathcal{F}_{\A})_{I(X)}$ is a free resolution of $S_{I(X)}/J(\A_X)S_{I(X)}$, but it is not minimal. This implies that for all $i\ge0$ $$b_i(S_{I(X)}/J(\A_X)S_{I(X)})\le b_i(S/J(\A)).$$

Let $\mathcal{F}_{\A_X}$ be a minimal free resolution of $S/J(\A_X)$. Similarly to the case of $S/J(\A)$, we have that $(\mathcal{F}_{\A_X})_{I(X)}$ is a free resolution of $S_{I(X)}/J(\A_X)S_{I(X)}$. Since $I(X)=(x_1\dots,x_s)$, we have that $Q(\A_X)\in\KK[x_1\dots,x_s]$. This implies that all of the matrix entries of $\mathcal{F}_{\A_X}$ belong to the ideal $I(X)=(x_1\dots,x_s)$, and hence $(\mathcal{F}_{\A_X})_{I(X)}$ is also minimal. This implies that for all $i\ge0$ $$b_i(S/J(\A_X))=b_i(S_{I(X)}/J(\A_X)S_{I(X)})\le b_i(S/J(\A)).$$%Then the algebraic total Betti numbers of $S/J(\A_X)$ are smaller than  the algebraic total Betti numbers of $S/J(\A)$.

The second inequality follows directly from the first one and Lemma~\ref{lemma:fromresDAtoJac}.
\end{proof}

As a direct consequence of Proposition~\ref{prop:localBettinumbers}, we have the following result

\begin{Corollary}\label{cor:projdimdisequal} Let $\A$ be an arrangement in $\KK^l$ and $X\in L(\A)$. Then $$\projdim(S/J(\A_X))\le\projdim(S/J(\A)).$$ This is equivalent to
$$\projdim(D(\A_X))\le\projdim(D(\A)).$$
\end{Corollary}

\begin{Remark} Corollary~\ref{cor:projdimdisequal} gives a different proof of Theorem~\ref{theo:localisfree}, i.e.\ the fact that if $\A$ is free, then $\A_X$ is free for any $X\in L(\A)$, with respect to the one that appears in \cite{orlterao}. 
\end{Remark}

If we also assume that $\A$ is a plus-one generated arrangement, Proposition~\ref{prop:localBettinumbers} gives us the following result
\begin{Corollary}\label{cor:l+1generatorlocal} Let $\A$ be a plus-one generated arrangement in $\KK^l$ and $X\in L(\A)$. Then $\A_X$ is free or $D(\A_X)$ is generated by $l+1$ vector fields.
\end{Corollary}
\begin{proof} If $\rk(X)=1$, then $\A_X=\{H\}$, for some $H\in\A$, and hence it is free. Assume $\rk(X)\ge2$. Since $\dim(S/J(\A_X))=2$, this implies that $\projdim(S/J(\A_X))\ge2$. On the other hand, by Corollary~\ref{cor:projdimdisequal}, $\projdim(S/J(\A_X))\le\projdim(S/J(\A))=3$. This implies that, $2\le\projdim(S/J(\A_X))\le3$. If $\projdim(S/J(\A_X))=2$, then $\A_X$ is free. Assume that $\projdim(S/J(\A_X))=3$. By Proposition~\ref{prop:localBettinumbers}, $1\le b_3(S/J(\A_X))\le b_3(S/J(\A))=1$. Moreover, by \cite[Corollary 20.13]{eisenbud}, the alternating sum of the Betti numbers of $S/J(\A_X)$ is zero, and hence 
$S/J(\A_X)$ has a minimal free resolution of the form
$$0\to S\to S^{\beta} \to S^{\beta}\to S\to S/J(\A_X)\to 0,$$
where $\beta=\codim(S/I(X))$. Similarly to Lemma~\ref{lemma:fromresDAtoJac}, this implies that $D(\A_X)$ has a minimal free resolution of the form
$$0 \to S \to S^{l+1} \to D(\A_X)\to 0,$$
and hence, $D(\A_X)$ is generated by $l+1$ vector fields.
\end{proof}

\begin{Lemma}\label{lemma:alphanotinIX} Let $\A$ be a plus-one generated arrangement in $\KK^l$, $X\in L(\A)$ and $\alpha\in S$ the linear form in the resolution~\eqref{eq:respogDA}. If $\alpha\notin I(X)$, then $\A_X$ is free.
\end{Lemma}
\begin{proof}
%Assumec$\alpha\notin I(X)$. 
Let $\mathcal{F}_{D(\A)}$ be a minimal free resolution of $D(\A)$ of the form \eqref{eq:respogDA}.
Since the localization preserves freeness, if we localize $\mathcal{F}_{D(\A)}$ at the prime ideal $I(X)$ we obtain $(\mathcal{F}_{D(\A)})_{I(X)}$ a resolution of $D(\A)_{I(X)}$. This implies that
$b_0(D(\A)_{I(X)})\le l+1$ and hence $D(\A)_{I(X)}$ can be generated by $l+1$ elements. Since $\alpha\notin I(X)$, the localization of the map $(\alpha,f_1,\dots,f_l)$ contains an invertible element (i.e. the localization of $\alpha$), and hence $D(\A)_{I(X)}$ can be generated by $l$ elements. This implies that $b_0(D(\A)_{I(X)})= l$.

By the proof of Proposition~\ref{prop:localBettinumbers}, $b_i(S_{I(X)}/J(\A)S_{I(X)})=b_i(S/J(\A_X))$. This implies that $b_i(D(\A)_{I(X)})=b_i(D(\A_X))$. In particular, if we consider $i=0$, we obtain that $b_0(D(\A_X))=l$, and hence $\A_X$ is free.
\end{proof}

We are now ready to prove the main result of this section.
\begin{Theorem}\label{theo:localofplusoneisfreeplusone} Let $\A$ be a plus-one generated arrangement in $\KK^l$ and $X\in L(\A)$. Then $\A_X$ is free or plus-one generated.
\end{Theorem}
\begin{proof} Let $\mathcal{F}_\A$ be a minimal free resolution of $S/J(\A)$ of the form \eqref{eq:respogjac}. %Moreover, we have that the map 
%$$\partial_3\colon S(-n-d)\to S(-n-d+1)\oplus(\bigoplus_{i=k}^{l} S(-n-d_i+1))$$
%is defined by a matrix of the form $(\alpha, f_1,\dots,f_{l-k+1})$, where $\alpha,f_i\in S$ and $\alpha$ is a homogenous polynomial of degree $1$.

If $X\in L(\A)$ is such that $\A=\A_X$ or if $\A_X$ is free, the statement is obviously true. Assume $\A_X\subsetneq\A$ and that $\A_X$ is non-free. By Corollary~\ref{cor:projdimdisequal}, $\projdim(S/J(\A_X))=3$. In this situation, similarly to the proof of Corollary~\ref{cor:l+1generatorlocal}, we have that $S/J(\A_X)$ has a minimal graded free resolution $\mathcal{F}_{\A_X}$ of the form
$$0{\to}S(-m-e){\to} S(-m-d'+1)\oplus(\bigoplus_{i=r}^{l} S(-m-d'_i+1)) {\to} S(-m+1)^{l-r+2}{\to}S,$$
where $m=|\A_X|$, $d'\le e$ and $k\le r$. To conclude we need to show that $d'=e$.
Notice that this is equivalent to prove that the map 
$$\partial'_3\colon S(-m-e)\to S(-m-d'+1)\oplus(\bigoplus_{i=r}^{l} S(-m-d'_i+1)) $$
is defined by a matrix of the form $(\alpha', f'_1,\dots,f'_{l-r+1})$, where $\alpha', f'_i\in S$ and $\alpha'$ is zero or a homogenous polynomial of degree $1$.

%Without loss of generalities, we can make a change of coordinates and assume that $I(X)=(x_1\dots,x_s)$ for some $1\le s\le l-1$.
By the proof of Proposition~\ref{prop:localBettinumbers}, $b_i(S_{I(X)}/J(\A_X)S_{I(X)})=b_i(S/J(\A_X))$. Hence, $$\projdim(S/J(\A))=\projdim(S/J(\A_X))=\projdim(S_{I(X)}/J(\A_X)S_{I(X)}).$$ This implies that the localization of the map $\partial_3$ in $\mathcal{F}_\A$ is not the zero map and it is defined by the localization of the matrix  $(\alpha, f_{i_1},\dots,f_{i_{l-k+1}})$.

By Lemma~\ref{lemma:alphanotinIX}, we can assume that $\alpha\in I(X)$. The localization of the map $\partial_3$ is defined by a matrix of the form $(\tilde{\alpha}, \tilde{f}_{i_1},\dots, \tilde{f}_{i_{l-k+1}})$, with $\tilde{\alpha}$ not invertible. This implies that the localization of the map $\partial'_3$ is represented by a matrix of the form $(\tilde{\alpha}, \tilde{f}_{j_1},\dots, \tilde{f}_{j_{l-r+1}})$, that is a submatrix of the matrix $(\tilde{\alpha}, \tilde{f}_{i_1},\dots, \tilde{f}_{i_{l-k+1}})$ obtained by deleting some invertible entries.
As described in the proof of Proposition~\ref{prop:localBettinumbers}, we can assume that all of the matrix entries of $\mathcal{F}_{\A_X}$ belong to the ideal $I(X)$. This implies that the map $\partial'_3$ is defined by a matrix of the form $(\alpha', f'_1,\dots,f'_{l-r+1})$, where $\alpha',f'_i\in S$ and $\alpha'$ is zero or a homogenous polynomial of degree $1$.
\end{proof}

One of the possible next step in this study would be to understand if we can generalize Theorem~\ref{theo:delisfreepog} to plus-one generated arrangements. At the moment this is still an open problem. The next example shows that in general the deletion of a plus-one generated arrangement can be free, plus-one generated or none of them.
\begin{Example}
Consider the arrangement $\A$ in $\CC^5$ defined by $xyztw(x+y+z)(x-w)(y+2t)$. $\A$ is plus-one generated, in fact $S/J(\A)$ has resolution
$$0\to S(-10)\to S(-9)^5\to S(-7)^5\to S.$$
Consider $\A_1=\A\setminus\{x+y+z=0\}$. Then $\A_1$ is free, in fact $S/J(\A_1)$ has resolution
$$0\to S(-8)^2\oplus S(-7)^2\to S(-6)^5\to S.$$
Consider $\A_2=\A\setminus\{y+2t=0\}$. Then $\A_2$ is plus-one generated, in fact $S/J(\A_2)$ has resolution
$$0\to S(-9)\to S(-8)^4\oplus S(-7)\to S(-6)^5\to S.$$
Consider $\A_3=\A\setminus\{x=0\}$. Then $\A_3$ is not free nor plus-one generated, in fact $S/J(\A_3)$ has resolution
$$0 \to S(-10) \to S(-9)^4 \to S(-8)^7 \to S(-6)^5 \to S.$$
\end{Example}

\bigskip
\paragraph{\textbf{Acknowledgements}} The authors would like to thank M. Yoshinaga and T. N. Tran for many helpful discussions. During the preparation of this article the second author was supported by JSPS Grant-in-Aid for Early-Career Scientists (19K14493). 

%---------------------------------------------%

%\bibliography{bibliothesis}{}

\begin{thebibliography}{10}

\bibitem{abe2018plus}
T.~Abe.
\newblock Plus-one generated and next to free arrangements of hyperplanes.
\newblock {\em To appear in {I}nternational {M}athematical {R}esearch
  {N}otices}, 2019.

\bibitem{Gin-freearr}
A.M. Bigatti, E.~Palezzato, and M.~Torielli.
\newblock New characterizations of freeness for hyperplane arrangements.
\newblock {\em {J}ournal of {A}lgebraic {C}ombinatorics}, 51(2):297--315, 2020.

\bibitem{dimca2017hyperplane}
A.~Dimca.
\newblock {\em Hyperplane Arrangements}.
\newblock Springer, 2017.

\bibitem{dimca2015nearly}
A.~Dimca and G.~Sticlaru.
\newblock Nearly free divisors and rational cuspidal curves.
\newblock {\em arXiv:1505.00666}, 2015.

\bibitem{eisenbud}
D.~Eisenbud.
\newblock {\em Commutative algebra with a view toward algebraic geometry},
  volume 150.
\newblock Springer, 1995.

\bibitem{orlterao}
P.~Orlik and H.~Terao.
\newblock {\em Arrangements of hyperplanes}, volume 300 of {\em Grundlehren der
  Mathematischen Wissenschaften [Fundamental Principles of Mathematical
  Sciences]}.
\newblock Springer-Verlag, Berlin, 1992.


\bibitem{palezzato2018free}
E.~Palezzato and M.~Torielli.
\newblock Free hyperplane arrangements over arbitrary fields.
\newblock {\em To appear in the {J}ournal of {A}lgebraic {C}ombinatorics},
  2019.
  
\bibitem{LefschetzPalez}
E.~Palezzato and M.~Torielli.
\newblock Lefschetz properties and hyperplane arrangements.
\newblock {\em {J}ournal of {A}lgebra}, 555:289--304, 2020.

\bibitem{palezzato2018hyperplane}
E.~Palezzato and M.~Torielli.
\newblock Hyperplane arrangements in {C}o{C}o{A}.
\newblock {\em {J}ournal of {S}oftware for {A}lgebra and {G}eometry},
  9(1):43--54, 2019.

\bibitem{saito}
K.~Saito.
\newblock Theory of logarithmic differential forms and logarithmic vector
  fields.
\newblock {\em J. Fac. Sci. Univ. Tokyo Sect. IA Math.}, 27(2):265--291, 1980.

\bibitem{terao1980arrangementsI}
H.~Terao.
\newblock Arrangements of hyperplanes and their freeness {I}.
\newblock {\em J. Fac. Sci. Univ. Tokyo Sect. IA Math.}, 27(2):293--312, 1980.

\end{thebibliography}

\bibliographystyle{plain}

\end{document}